\documentclass[10pt,reqno]{amsart}
\usepackage{amsmath,amssymb,verbatim}
\usepackage[utf8]{inputenc}
\pdfoutput=1
\usepackage{enumerate}
\usepackage{pdflscape}
\usepackage{amsthm}
\usepackage{mathrsfs}
\usepackage{color}
\usepackage{multicol}
\usepackage[normalem]{ulem}
\usepackage{cancel}
\usepackage{tikz}
\usetikzlibrary{matrix}
\usepackage[all]{xy}
\usepackage{mathtools}
\usepackage{arydshln}
\usepackage[shortlabels]{enumitem}
  \usepackage[english]{babel}
\usepackage{hyperref}
\hypersetup{
    colorlinks = true,
    linkbordercolor = {white},
    allcolors = blue
} 
\usepackage[capitalize, noabbrev]{cleveref}
 
\makeatletter
\@namedef{subjclassname@1991}{\textup{2020} Mathematics Subject Classification}
\makeatother

\topskip=\baselineskip 
\newtheorem{theorem}{Theorem}[section]
\newtheorem{lemma}[theorem]{Lemma}

\newtheorem{proposition}[theorem]{Proposition}

\newtheorem{examples}[theorem]{Examples}
\newtheorem{remarks}[theorem]{Remarks}

\newtheorem{question}[theorem]{Question}

\newtheorem*{theorem*}{Theorem}

\newtheorem{maintheorem}{Theorem}

\newtheorem{maincorollary}[maintheorem]{Corollary}
\newtheorem{mainproposition}[maintheorem]{Proposition}

\newtheorem{mainquestion}[maintheorem]{Question}

\makeatletter
\@addtoreset{equation}{section}
\makeatother

\newcommand{\M}{{\mathrm{D}}}

\newcommand{\Z}{{\mathbb Z}}

\newcommand{\F}{{\mathbb F}}

\newcommand{\GEN}[1]{\left\langle #1 \right\rangle}

\newcommand{\aug}[1]{\mathrm{I}(#1)}
\newcommand{\augNor}[2]{\mathrm{I}(#1;#2)}

\newcommand{\ZZ}{\mathrm{Z}}

\newcommand{\qand}{\quad \text{and} \quad}

\DeclareMathOperator{\dG}{d}

\title[The modular isomorphism problem and abelian direct factors]{The modular isomorphism problem and abelian direct factors}

\author{Diego Garc\'{\i}a-Lucas }
 \address{Departamento de Matem\'aticas, Universidad de Murcia, Spain}
 \email{diego.garcial@um.es}

 \thanks{The  author   is partially supported by Grant PID2020-113206GB-I00 funded by MCIN/AEI/10.13039/501100011033. }



\date{\today}

\begin{document}

\begin{abstract}
Let $p$ be a   prime and let $G$ be a finite $p$-group. We show that the isomorphism type of the maximal abelian direct factor of $G$, as well as the isomorphism type   of the group algebra over $\F_p$ of the non-abelian remaining direct factor, are determined by $\F_p G$, generalizing the main result in \cite{MSS21} over the prime field. In order to do this, we address the problem of finding characteristic subgroups of $G$ such that their relative augmentation ideals depend only on the $k$-algebra structure of $kG$, where $k$ is any field of characteristic $p$, and relate it to the modular isomorphism problem,   extending and reproving some known results.

\end{abstract}

\maketitle

\section{Introduction}

Let $k$ be a field of characteristic $p$, and $G$ and $H$ finite $p$-groups. The {}\emph{modular isomorphism problem} (MIP) asks whether the existence of an isomorphism of $k$-algebras $kG\cong kH$  implies the existence of an isomorphism of groups $G\cong H$. The most classical version of this question also assumes that $k=\F_p$, the field with $p$ elements. Indeed, it is already   mentioned by Brauer in \cite{Bra63}  
 as a   possibly   \emph{much easier} 
 particular case of the general isomorphism problem for group rings (\cite[Problem 2]{Bra63}). Despite Brauer's optimistic observation,   only  partial positive results for MIP have been obtained  under   quite severe  restrictions on the structure of $G$ and $H$. For example, MIP is known to have positive answer if the groups are abelian (Deskins  \cite{Deskins1956}), metacyclic (Bagi\'{n}ski  \cite{BaginskiMetacyclic}, and Sandling  \cite{San96}),  or have class $2$ and   elementary abelian derived subgroup (Sandling  \cite{San89}). Some more recent positive results and  approaches can be found in \cite{Sakurai,MM20,BdR20,MS22}, and an up to date state of the art, in \cite{Mar22}.    The modular isomorphism problem  is now known to have negative answer in general, as it is shown in \cite{GarciaMargolisdelRio} that there exist  non-isomorphic finite $2$-groups with isomorphic group algebras over every field of characteristic $2$. This counterexample makes it even more interesting to investigate which properties (weaker than the isomorphism type) of $G$ and $H$ must necessarily coincide when $kG\cong kH$. In any case, the original problem remains open for $p$ an odd prime.

On the other hand, and despite all the attention MIP received, an approach that surprisingly seems    to not have been exploited until very recently is to reduce the problem from the vast class of all finite $p$-groups to some smaller (but maybe as vast and complicated) subclass of groups.   In \cite{MSS21}  it is shown that the modular isomorphism problem  can be reduced to the same problem over groups without elementary abelian direct factor. We generalize this result by dropping the `elementary' hypothesis, i.e., showing that MIP can be reduced to the problem over groups without abelian  direct factors, with no restrictions on the exponent. 
 We formalize this as follows:  
For a finite $p$-group $G$, consider a decomposition $G=\text{\rm Ab}(G)\oplus \text{\rm{NAb}(G)}$, where $\text{\rm Ab}(G)$ and $ \text{\rm{NAb}(G)}$ are subgroups of $G$ such that $\text{\rm Ab}(G)$ is abelian, and maximal such that a decomposition like that exists. From the Krull-Remak-Schmidt theorem it follows that the isomorphism types of
$\text{\rm Ab}(G)$ and $ \text{\rm{NAb}(G)}$ do not depend on the chosen decomposition, so they are group-theoretical invariants of $G$. Our main theorem states that we can disregard the direct factor $\text{Ab}(G)$ in the study of MIP. Formally:
\begin{maintheorem}\label{maintheorem} Let $k=\F_p$ and   $G$ and $H$ be finite $p$-groups. Then  
	\begin{equation*}
	kG\cong kH\qquad \text{if and only if}\qquad 	k\left(\textrm{\rm NAb}(G)\right)\cong k\left( \textrm{\rm NAb}(H)\right) \qand  \textrm{\rm Ab}(G)\cong\textrm{\rm Ab}(H).
	\end{equation*}
\end{maintheorem}

As an immediate corollary, we can extend non-trivially some of the classes of groups for which MIP is known to have a positive answer. We summarize some of them in the following corollary.

\begin{maincorollary}\label{maincorollary}
	Let $k=\F_p$ and $A$ and $G$ be finite $p$-groups such that $A$ is abelian and at least one of the following holds:\begin{enumerate}
		\item \label{maincorollary1} $G$ is metacyclic.
		\item $G$ is $2$-generated of nilpotency class $2$.
		\item $G$ is $2$-generated with nilpotency class $3$ with elementary abelian derived subgroup.
		\item $G$ is elementary abelian-by-cyclic.
		\item $G$ has a cyclic subgroup of index $p^2$.
		\item The third term of the Jennings series of $G$ is trivial.
		\item \label{maincorollary7}  The fourth term of the Jennings series of $G$ is trivial and $p>2$. 
		\item \label{maincorollary8} $G$ has order at most $p^5$.
	\end{enumerate}
If $H$ is another group such that $kH\cong k(G\oplus A)$, then $H\cong G\oplus A$.
\end{maincorollary}

Observe that none of these properties is closed under taking direct products with abelian groups, so our result is indeed non-trivial. The primality of the field is important in the proof of \Cref{maintheorem}; however, for an arbitrary field of characteristic $p$ we can still recover the isomorphism type of the maximal abelian direct factor.
\begin{mainproposition}\label{mainprop}
	Let $G$ and $H$ be finite $p$-groups and $k$ be a field of characteristic $p$. Then $kG\cong kH$ implies that $\text{Ab}(G)\cong \text{Ab}(H)  $.
\end{mainproposition}
 
Hence naturally  the following question arises
\begin{mainquestion}
	Let $k$ be a field of characteristic $p$, and $G$ and $H$ be finite $p$-groups.  Does $kG\cong kH$ imply  that 
	$  k( \text{\rm NAb}(G))\cong k(\text{\rm NAb}(H))  $?
	
	\end{mainquestion}

We recall that the \emph{indecomposable decomposition} of a finite group $G$, i.e., the  indecomposable direct factors of $G$ arising from the Krull-Remak-Schmidt theorem, is unique up to isomorphism and reordering. Since the groups in \cite{GarciaMargolisdelRio} are indecomposable, 
 it makes also sense to ask if \Cref{maintheorem} can be extended so that   the modular isomorphism problem is reduced to the same problem over indecomposable finite $p$-groups, carrying the reduction approach (in the sense of direct factor decompositions) to its ideal conclusion. In other words:
 \begin{mainquestion}
 	Let $k=\F_p$ and $G$ and $H$ be finite $p$-groups. Does $kG\cong k H$  imply that the terms of the indecomposable decompositions of $G$ and $H$   have pairwise isomorphic group algebras over $k$?
 \end{mainquestion} 
\Cref{lemma:semidirect2} seems to suggest that this  problem is not completely hopeless; however the remainder of techniques used in \Cref{section:proofs} rely heavily on the fact that the direct factors we disregard are abelian (and hence contained in the center of the group), so new ideas and techniques would be needed to follow this path.  
 
 \vspace*{0.2cm}

The paper is organized as follows. In \Cref{SectionPreliminaries} we set  notation and present some well-known results related to the modular isomorphism problem. In \Cref{section:transfer} we address the problem of finding normal subgroups $N$ of $G$ such that their relative augmentation ideals depend only on the  algebra structure of $kG$, and relate this problem to the MIP.  Finally, in \Cref{section:proofs} we   prove, with the help of the results in the previous section, both \Cref{mainprop} and \Cref{maintheorem}. 

\section{Notation and preliminaries}\label{SectionPreliminaries}

Throughout the paper, $p$ will denote a prime number, $k$ a field of characteristic $p$, $\F_p$ the field with $p$ elements and $G$ and $H$ finite $p$-groups. The  group algebra of $G$ over $k$ is denoted by $kG$ and its augmentation ideal  is denoted by $\aug{G}$. It is a classical result that $\aug{G}$ is also the Jacobson ideal of $kG$.  For every normal subgroup $N$ of $G$, we write $\augNor{N}{G}$ for the relative augmentation ideal $\aug{N}{kG}$,  i.e., the two-sided  ideal generated by the elements of the form $n-1$ with $n\in  N$. It is well-known that this ideal is just the kernel of the natural projection $kG\to k(G/N)$, i.e., the homomorphism of algebras extending the natural projection $G\to G/N$. Moreover, $\augNor{N}{G}\cap (G-1)= N-1$.  For these basic facts about the augmentation ideals we refer to \cite[Section 1.1]{Pas77}.
 We denote by $[kG,kG]$ the vector subspace of $kG$ generated by the elements of the form $xy-yx$, with $x,y\in kG$. Given a subspace $U$ of a vector space $V$ over $k$, we write $\text{\rm codim}_V(U)=\dim(V)-\dim(U)$ to denote the codimension of $U$ in $V$.

Our group theoretic notation  is mostly standard. We write $\oplus$ both for internal and external direct products of groups, and also for the direct sum of vector spaces. For $n \geq 1$, we denote by $C_n$   the cyclic group of order $n$. Given $n,r\geq 1$, a   \emph{homocyclic} group of exponent $n$ and rank $r$ is a group isomorphic to  $(C_{n})^r=C_n\oplus\dots \oplus C_n$ ($r$ times).  We  let $|G|$ denote the order of the finite $p$-group $G$, $\ZZ(G)$ its center, $\{\gamma_i(G)\}_{i\geq 1}$ its lower central series, $G'=\gamma_2(G)$  its commutator subgroup, $\Phi(G)$ its Frattini subgroup and  $\dG(G)=\min \{|X|   :  X\subseteq G \textup{ and } G=\GEN{X}\}$ its minimum number of generators. It is well known that $G/\Phi(G)$ is an elementary abelian $p$-group, and  $d(G)=d(G/\Phi(G))$. If $A$ is a homocyclic $p$-group, then $d(A)$ equals the rank of $A$. If moreover $A$ is  elementary abelian, it can be seen as a vector space  over $\F_p$ with dimension $d(A)$. A \emph{Burnside basis} of $G$ is a minimal set of generators of $G$, i.e., a subset of $G$ such that its image in $G/\Phi(G)$ is a basis as vector space.    We   define the \emph{omega} series $(\Omega_n(G))_{n\geq 0}$ and the \emph{agemo} series $(\mho_n(G))_{n\geq 0}$ of $G$ by:
$$\Omega_n(G)=\GEN{g\in G : g^{p^n}=1} \qand 
\mho_n(G)=\GEN{g^{p^n} : g \in G}.$$   
If $N$ is a normal subgroup of $G$ and $n\geq 1$,   we also write
$$\Omega_n(G:N)=\GEN{g\in G: g^{p^{n}}\in N}, $$ that is, $\Omega_n(G:N)$ is the unique subgroup of $G$ containing $N$ such that 
$ \Omega_n(G:N)/N=\Omega_n(G/N). $ 
   The \emph{Jennings series} $(\M_n(G))_{n\geq 1}$ of $ G$ is defined by 
\begin{equation*}
\label{EqJenningsLazard}
\M_{n}(G) = \{g\in G : g-1\in \aug{G}^n\}= \prod_{ip^j\ge n}   \mho_j(\gamma_i(G)) 
\end{equation*}  In particular $\M_1(G)=G$ and $\M_2(G)=\mho_1(G)G'=\Phi(G)$.  A property of these series that we will use is that if $G$ is abelian, then the orders of the terms completely determine the isomorphism type of $G$.  
 For more on the Jennings series, see for instance \cite[Section~11.1]{Pas77} and \cite[Section~III.1]{Seh78}.

 The next proposition collects some well-known results about MIP that will have some relevance for our results; a complete list  can be consulted in \cite{Mar22}.

 \begin{proposition}\label{prop:knowncases}
 	Let $k=\F_p$ be the field with $p$ elements, and let $G$ and $H$ be finite $p$-groups. Suppose that $G$ satisfies at least one of the following conditions:
 	\begin{enumerate}
 		\item $G$ is abelian \cite{Deskins1956}.
 		\item $ G$ is metacyclic \cite{BaginskiMetacyclic, San96}.
 		\item $G$ is $2$-generated with nilpotency class $2$   \cite{BdR20}.
 			\item $G$ is $2$-generated with elementary abelian derived subgroup and nilpotency class $3$ \cite{MM20,Bag99}.
 			\item $G$ has a cyclic subgroup of index $p^2$ \cite{BK07}.
 		\item $G$ is elementary abelian-by-cyclic \cite{Bag99}.
 		\item $\M_3(G)=1$ \cite{PS72}.
 	\item $\M_4(G)=1$ and $p>2$ \cite{Her07}.
 	\item $|G|\leq p^5$ \cite{Passman1965p4,SalimSandling1995}.
 	\end{enumerate}
 If $kG\cong kH$, then $G\cong H$.
 \end{proposition}

  Let $\phi:kG\to kH$ be an isomorphism preserving the augmentation. Then the ideal of $kG$ generated by $[kG,kG]$ is exactly $\augNor{G'}{G}$. As $\phi([kG,kG])=[kH,kH]$, it follows that 
  \begin{equation}
  	\phi(\augNor{G'}{G})=\augNor{H'}{H}. \label{DerivedCanonical}
  \end{equation} 
  This is a classical result that already appears in \cite{Coleman1962}, and plays a role in the proofs of a great part of the known results about MIP (e.g., the ones involving the small group algebra such as \cite{San89}).  
%The following lemma is well-known: the first identity was already used in \cite{Coleman1962}, and plays a role in the proofs of a great part of the known results about MIP (e.g., the ones involving the small group algebra such as \cite{San89}), while the second identity is used in the proof of \cite[Theorem 2(ii)]{BK07}.
%
%
%\begin{lemma}\label{lemma:knowncanonical}
%	Let $k$ be a field of characteristic $p$ and $G$ and $H$ finite $p$-groups. If   $\phi:kG\rightarrow kH$ is an isomorphism of $k$-algebras preserving the augmentation 
%	then  $$\phi(\augNor{G'}{G})=\augNor{H'}{H} \ \textup{ and } \ \phi(\augNor{\ZZ(G)G'}{G})=\augNor{\ZZ(H)H'}{H}.$$
%\end{lemma} 
In  \Cref{section:transfer} we will be  interested in subgroups satisfying the property in \eqref{DerivedCanonical}. More concretely, given a family $\mathcal F$ of groups, we say that a map $N_*:\mathcal F\to \mathcal C$, where $\mathcal C$ is the class of all groups, is a \emph{subgroup assignation}  if given a group $G$ in $\mathcal F$ it  returns a  normal subgroup $N_G$ of $ G$.   Partially following \cite[Section 3.1.2]{SalimTesis}, we say that the subgroup assignation $N_*$ is    \emph{$k$-canonical} over $\mathcal F$, if 
$$\phi\left(\augNor{N_G}{G}\right)=\augNor{N_H}{H}  $$
for $G,H\in \mathcal F$ and every isomorphism $\phi:kG\to kH$ preserving the augmentation. Clearly if a subgroup assignation is $k$-canonical, then the image subgroup is a characteristic subgroup of the original group. Unless stated otherwise, in the rest of this paper all the $k$-canonical assignations will be over the family $\mathcal F$ of all the finite $p$-groups. In the next section we provide specific ways to obtain group-theoretical invariants of $G$ determined by $kG$ from $k$-canonical assignations, and also to obtain new $k$-canonical assignations from the known ones.

\section{The $k$-canonical subgroup problem and a transfer lemma}\label{section:transfer}

The following general fact is widely known, as it is mentioned in the proof of \cite[Lemma 6.10]{Sandling85}. A particular case of  it  is  stated and reproved in \cite[Lemma 2.7]{MSS21} restricted to $L=\ZZ(G)$ and $N=\mho_t(G)G'$ for any positive integer $t$, though their proof works in general. %Since the proof is not different, but for the sake of completeness we display it. 
 \begin{lemma}\label{lemma:MSS21}
	Let $k$ be a field of characteristic $p $ and $G$ a $p$-group. If $N$ and $L$ are normal subgroups of  $G$, then
	$
		\augNor{L}{G}\cap \aug{N}=\augNor{L\cap N}{N}.$
\end{lemma}

All the items in the following lemma seem  to be known, at least for some specific $k$-canonical subgroup assignation (mostly $G\mapsto G'$), but to the best of our knowledge it   had never been stated in general. 
 \begin{lemma}\label{JenningsNor}
	Let $k$ be a field of characteristic $p $, let $G$ and $H$ be  finite $p$-groups and let $L_G$  and $L_H$ be normal subgroups of $G$ and $H$,  respectively. Assume that there is an isomorphism $\phi:kG\to kH$ such that  $ \phi\left(\augNor{L_G}{G} \right)=\augNor{L_H}{H}$. Then
\begin{enumerate}
	\item \label{JenningsNor1} For each $i\geq 1$, one has that $
		\M_{i}(L_G)/\M_{i+1}(L_G) \cong \M_{i}(L_H)/\M_{i+1}(L_H) $.
\item\label{JenningsNor1+1/2} If $L_G$ and $L_H$ are both abelian, then  $
	L_G\cong L_H$.
\item \label{JenningsNor2-1/2}
If $N_\Gamma$ is a subgroup of $\Gamma$ containing $\Gamma'$ for $\Gamma\in \{G,H\}$ such that $\phi(\augNor{N_G}{G})=\augNor{N_H}{H}$, then $ G/L_GN_G\cong G/L_HN_H $ and $ L_GN_G/N_G\cong L_HN_H/N_H$. 
\item\label{JenningsNor2}   $
 G/L_GG'\cong H/L_HH'$ and $
  	L_GG'/G'\cong L_HH'/H'$. 
\item \label{JenningsNor3} $\ZZ(G)\cap L_G\cong \ZZ(H)\cap L_H$ and $ \ZZ(G)L_G/L_G\cong \ZZ(H)L_H/L_H$.
\end{enumerate} 
\end{lemma}

\begin{proof}   \eqref{JenningsNor1} is proven for $k=\F_p$ in \cite[Lemma~6.26]{Sandling85}, although  the proof also works for an arbitrary field $k$ of characteristic $p$. See also the proof of \cite[Lemma~2]{BaginskiMetacyclic} for $N=G'$.  \eqref{JenningsNor1+1/2} is an immediate consequence of \eqref{JenningsNor1}, since the isomorphism type of an abelian group is determined by the orders of the terms of its Jennings series.  As for \eqref{JenningsNor2-1/2}, %it is also surely well-known, but we are not aware fo the existence of any explicit published  proof, so for completeness we include one. L
	let $\phi:kG\to kH$ an isomorphism of $k$-algebras. Then%, as both subgroups are normal, 
	$$\phi\left(\augNor{L_GN_G}{G}\right)=\phi(\augNor{L_G}{G}+\augNor{N_G}{G})=\augNor{L_H}{H}+\augNor{N_H}{H} =\augNor{L_HN_H}{H}.$$
	This implies that $\phi$ induces an isomorphism $k(G/L_GN_G)\cong k(H/L_HN_H)$, thus the  isomorphism $G/L_GN_G\cong H/L_HN_H$ follows from \eqref{JenningsNor1+1/2} because both groups are abelian, as $\Gamma'\subseteq N_\Gamma$. 
	For the last isomorphism in \eqref{JenningsNor2-1/2}, observe that
	$\phi$ also induces an isomorphism $\hat\phi :k(G/N_G)\to k(H/N_H)$. Moreover, if $\pi_{N_\Gamma}:k\Gamma\to k(\Gamma/N_\Gamma)$ is the natural projection, then $\pi_{N_\Gamma}\left(\augNor{L_\Gamma N_\Gamma }{\Gamma}\right)=\augNor{L_\Gamma N_\Gamma/N_\Gamma}{\Gamma/N_\Gamma}$. Thus $\hat \phi(\augNor{L_GN_G/N_G}{G/N_G})=\augNor{L_HN_H/N_H}{H/N_H}$. Then applying item \eqref{JenningsNor1} it follows that $$\M_{i}(L_G N_G/N_G)/\M_{i+1}(L_GN_G/ N_G) \cong \M_{i}(L_HN_H/N_H)/\M_{i+1}(L_HN_H/N_H) $$ for each $i\geq 1$. Hence, as the groups $L_GN_G/N_G$ and $L_HN_H/N_H$ are abelian, they must be isomorphic. Now \eqref{JenningsNor2} follows immediately using \eqref{DerivedCanonical}. 
	
	The underlying idea behind \eqref{JenningsNor3} is once again the proof of \cite[Lemma 6.10]{Sandling85}, which states the same result for $L_\Gamma=\Gamma'$; similar ideas are used in \cite{MSS21} to prove the result specialized to $L_G=\mho_t(G)G'$, for any positive integer $t$. It is well-known (see \cite[Lemma 6.10]{Sandling85}) that 
\begin{equation}\label{CenterDecomposition} \ZZ(kG)=\aug{\ZZ(G)}\oplus [kG,kG]\cap \ZZ(kG),
\end{equation}
where $[kG,kG]\cap \ZZ(kG)$ is an ideal of $\ZZ(kG)$ and $\aug{\ZZ(G)}$ is a $k$-algebra. Moreover \Cref{lemma:MSS21} for $N=\ZZ(G)$ and $L=L_G$ yields that  
	$  \aug{\ZZ(G)} \cap \augNor{L_G}{G} =\aug{\ZZ(G)\cap L_G}   $. Therefore 
	$$\ZZ(kG)\cap \augNor{L_G}{G}+[kG,kG]\cap\ZZ(kG)= \augNor{\ZZ(G)\cap L_G}{\ZZ(G)}\oplus [kG,kG]\cap \ZZ(kG)  . $$
	Since   $\phi:kG\to kH$ maps $[kG,kG]\cap \ZZ(kG) $ to $[kH,kH]\cap \ZZ(kH) $, we deduce that the restriction of $\phi$ to $\ZZ(kG)\cap \aug{G}$ induces an isomorphism $\tilde\phi:\aug{\ZZ(G)}\to \aug{\ZZ(H)}$ such that $ \tilde\phi\left(\augNor{\ZZ(G)\cap L_G}{\ZZ(G)}\right)=\augNor{\ZZ(H)\cap L_H}{\ZZ(H)} $; 
	hence the isomorphism $\ZZ(G)\cap L_G\cong \ZZ(H)\cap L_H$ follows from \eqref{JenningsNor1+1/2}. Moreover $\tilde \phi$ induces an isomorphism $k\left(\ZZ(G)/\ZZ(G)\cap L_G\right)\cong k\left(\ZZ(H)/\ZZ(H)\cap L_H\right)$, so the last isomorphism also follows from \eqref{JenningsNor1+1/2}.
\end{proof}

 \Cref{JenningsNor} gives an idea about how useful to obtain new group-theoretical invariants of the group algebra $kG$ is to find $k$-canonical subgroup assignations over the family of finite $p$-groups. This would be also  an interesting problem by itself, as it seems to be the most natural way to study how the normal subgroup structure of the group is reflected inside the group algebra.  
 
 Let us compare the situation with the case of group algebra with integral coefficients: Let $\Gamma_1$ and $\Gamma_2$ be two finite groups such that there is an isomorphism $\phi:  \Z  \Gamma_1\to \Z  \Gamma_2$ preserving the augmentation. We adopt temporarily (only for this paragraph)  the notations $\aug{\Gamma_i}$ and $\augNor{N } {\Gamma_i}$ for the augmentation ideal of $\Z \Gamma_i$ and the  augmentation ideal of $\Z \Gamma_i$ relative to a normal subgroup $N$ of $\Gamma_i$. By the so called \emph{Normal Subgroup Correspondence} in the integral case (see \cite[Theorem III.4.17]{Seh78}), $\phi$ induces an isomorphism $\phi^*$ between the  lattices of normal subgroups of $\Gamma_1$ and $\Gamma_2$. Furthermore, by \cite[Theorem III.4.26]{Seh78} this isomorphism satisfies 
 $ \phi(\augNor{N}{\Gamma_1})=\augNor{\phi^*(N)}{\Gamma_2}. $ 
 
 Back to the modular case, no such Normal Subgroup Correspondence exists in general. However, there exists a limited version of this correspondence,    restricted to the sublattice of the lattice of normal subgroups of $G$ formed by the $k$-canonical subgroups of $G$ (more properly, evaluations in $G$ of $k$-canonical assignations). This obvious correspondence is given by $\phi^*(N_G)=N_H$, for each $k$-canonical assignation $N_*$. Thus the real problem is to identify the $k$-canonical subgroups and determine how large this sublattice can be.  In general, one might ask the following

 \begin{question}
 	Given a field $k$   and a family $\mathcal F$ of groups, which are the $k$-canonical subgroup assignations $N_*:\mathcal F\to \text{Grp}$?
 \end{question}

  Unfortunately, if $\mathcal F$ is the family of finite $p$-groups, the list of subgroups assignations which are known  to be $k$-canonical  is limited to $G'$   and to some less known examples appearing in \cite{SalimTesis}  and \cite{BK07}, which we will mention later. Even more unfortunately, a choice of subgroup assignation as natural as the center of the group is known to fail to satisfy this property in general,  as shown by Bagi\'{n}ski and Kurdics \cite[Example 2.1]{BK19}, by virtue of a group $G$ of order $81$ and maximal class, for which  there exists an automorphism $\phi:\F_3G\to \F_3G$  such that $\phi\left(\augNor{\ZZ(G)}{G}\right)\neq \augNor{\ZZ(G)}{G}$. Since $G$ is of maximal class, also $\ZZ(G)=\gamma_3(G)$, so the terms of the lower central series in general (aside from $G'=\gamma_2(G)$) are neither candidates to be $k$-canonical subgroups assignations. 
 
  However, though limited by the existence of this counterexample, it is still possible that the search of this kind of subgroups could lead to new invariants and MIP-related results, or at least to shed some new light on the existing ones.  The remainder of the section is devoted to extend this list, and to that end  the following pair  of easy general facts will be useful. The first one corresponds to \cite[Proposition III.6.1]{Seh78}.

 \begin{lemma}\label{lemma:abelianppower}
 	Let $A$ be an abelian finite $p$-group, $k$ be a field of characterisitc $p$, $t $ be a positive integer,  and $\lambda :kA\to k\mho_t(A)$ be the homomorphism given by $x\mapsto x^{p^t}$. Then
 	$$\ker\lambda =\augNor{\Omega_t(A)}{A} .$$
 	Moreover, the $k$-linear hull of the image of $\lambda$ equals $k\mho_t(A)$. 
 \end{lemma}
\begin{proof}
Observe that the $p^t$-power map $A\to \mho_t(A)$, $x\mapsto x^{p^t}$ is a surjective homomorphism of groups whose kernel is $\Omega_t(A)$, so that there is an isomorphism of groups $\kappa:A/\Omega_t(A) \to \mho_t(A)$ which extends to an algebra isomorphism $\kappa:k\left(A/\Omega_t(A)\right)\to k\mho_t(A)$. Denote $\sigma:kA\to k\mho_t(A)$ the $k$-linear extension of the $p^t$-power map $A\to \mho_t(A)$, and $\pi:kA\to k\left(A/\Omega_t(A)\right)$ the canonical projection. Then $\sigma$ factors as $\sigma= \kappa \circ \pi $.  Now let $\tau :kA\rightarrow kA$  be the  ring  homomorphism  extending the identity on $A$ and the  map $x\rightarrow x^{p^t}$ on $k$.  Then $\lambda$ factors as $\lambda= \tau\circ \sigma=\tau\circ \kappa\circ \pi$. Since both $\tau$ and $\kappa$ are injective, $\ker \lambda=\ker\pi=\augNor{\Omega_t(A)}{A}$.  	For the last statement it suffices to observe that the image of $\lambda$ contains $\mho_t(A)$. 
\end{proof}

The following observation is trivial. 
\begin{lemma}\label{Triviality}Let $k$ be a field of characteristic $p$ and $G$ a finite $p$-group. 
	Let also $N$ and $L$ be normal subgroups of $G$ with $N\subseteq L$, and let $\pi_N:kG\to k(G/N)$ be the canonical projection. Then 
	$$\augNor{L}{G} =\pi_N^{-1}(\augNor{L/N}{G/N}). $$
\end{lemma}
 
We close this section with a  lemma that allow us to obtain new $k$-canonical subgroup assignations  from known ones,  and a series of examples connecting this  with other results. A version of this lemma (comprehending only the first item), appears in \cite{GLdRS2022} with the name of Transfer Lemma. We choose to reprove it here, as the method of proof is strongly related to the one of the other items.

\begin{lemma}[Transfer Lemma]\label{mainlemma}
	Let $p$ be a prime number and $G$ and $H$ be finite $p$-groups.
	Let $t$ be a positive integer and for $\Gamma\in \{G,H\}$ let $N_\Gamma$ and $L_\Gamma$ be  normal subgroups of $\Gamma$ such that $\Gamma'\subseteq N_\Gamma$. Let $k$ be a field of characteristic $p$ and $\phi:kG\rightarrow kH$ be  a $k$-algebra isomorphism preserving the augmentation   such that $\phi(\augNor{N_G}{G})=\augNor{N_H}{H}$ and $\phi(\augNor{L_G}{G})=\augNor{L_H}{H}$. Then 
	\begin{enumerate}
		\item  \label{mainlemma1}
		$  \phi(\augNor{\Omega_t(G:N_G)}{G})=\augNor{\Omega_t(H:N_H)}{H}.$ 
		\item  \label{mainlemma3}  $\phi(\augNor{\mho_t(L_G )N_G}{G})=\augNor{\mho_t(L_H )N_H}{H}.$
	\item \label{mainlemma2} $
			\phi\left(\augNor{\Omega_t(\ZZ(G))N_G}{G}\right)=\augNor{\Omega_t(\ZZ(H))N_H}{H}.$
		
	\end{enumerate}

\end{lemma}
\begin{proof}  Throughout the proof,  for a normal subgroup $K$ of $\Gamma$, we denote by $\pi_K:k\Gamma\to k(\Gamma/K)$ the natural projection, with kernel $\augNor{K}{\Gamma}$. The hypothesis $\phi\left(\augNor{N_G}{G}\right)=\augNor{N_H}{H}$ yields that the isomorphism $\tilde \phi:k(G/N_G)\to k(H/N_H)$ induced by $\phi$ makes the following square commutative
	\begin{equation}\begin{array}{c}
		\label{diagram1}\xymatrix{
			kG \ar[d]_-{\phi} \ar[r]^-{\pi_{N_G}} & k(G/N_G) \ar[d]^-{\tilde \phi} \\
			kH \ar[r]_-{\pi_{N_H}} & k(H/N_H)
		} 
\end{array}
	\end{equation} 
 	Let also $\lambda_\Gamma:k(\Gamma/N_\Gamma)\to k\mho_t(\Gamma/N_\Gamma)\subseteq k(\Gamma/N_\Gamma)$ be the $p^t$-power map, which is a ring homomorphism because of the commutativity of $\Gamma/N_\Gamma$.
There is a commutative square
	\begin{equation}\label{diagram2}\begin{array}{c}
		\xymatrix{
			k(G/N_G) \ar[d]_-{\tilde\phi}   \ar[r]^-{\lambda_G}& k\mho_t(G/N_G) \ar[d]^-{\tilde\phi} \\
			k(H/N_H)   \ar[r]_-{\lambda_H}& k\mho_t(H/N_H)  
		} 
\end{array}
	\end{equation}
  where the vertical arrow on the right is just the restriction of $\tilde \phi$.
	
	\eqref{mainlemma1} Taking the kernels of the maps $\lambda_\Gamma$, the commutativity of \eqref{diagram2} and \Cref{lemma:abelianppower} yield  that
$$\tilde\phi\left( \augNor{\Omega_t(G/N_G) }{G/N_G}\right)=\augNor{\Omega_t(H/N_H) }{H/N_H},$$
which can be rewritten as 
\begin{equation*}\label{eq:mainlemma1}
	\tilde\phi\left( \augNor{\Omega_t(G:N_G)/N_G }{G/N_G}\right)=\augNor{\Omega_t(H:N_H)/N_H }{H/N_H}.
\end{equation*}   Now,  since  by \Cref{Triviality} we have that
$  \pi_{N_\Gamma}^{-1} \left(\augNor{\Omega_t(\Gamma:N_\Gamma)/N_\Gamma}{\Gamma/N_\Gamma}\right)=\augNor{ \Omega_t(\Gamma:N_\Gamma)}{\Gamma} $, 
 the commutativity of \eqref{diagram1} yields  that $$\phi\left(\augNor{\Omega_t(G:N_G)}{G}\right)=\augNor{\Omega_t(H:N_H)}{H}.$$

	  \eqref{mainlemma3} Observe that the hypotheses imply that $\phi(\augNor{L_GN_G}{G})=\augNor{L_HN_H}{H}$; moreover in general $\mho_t(L_\Gamma)N_\Gamma=\mho_t(L_\Gamma N_\Gamma)N_\Gamma$. Thus     we can assume without loss of generality that $N_\Gamma\subseteq L_\Gamma$. Now observe that $\pi_{N_\Gamma}\left(\augNor{L_\Gamma}{\Gamma}\right)=\augNor{L_\Gamma/N_\Gamma}{\Gamma/N_\Gamma}$. Hence, by the commutativity of \eqref{diagram1}, 
	  $$\tilde \phi\left( \augNor{L_G/N_G}{G/N_G} \right)=\augNor{L_H/N_H}{H/N_H}. $$
	  Observe also that 
	  $ \lambda_\Gamma\left(\augNor{L_\Gamma/N_\Gamma}{\Gamma/N_\Gamma}  \right) $ generates $\augNor{\mho_t(L_\Gamma/N_\Gamma)}{\Gamma/N_\Gamma}=\augNor{\mho_t(L_\Gamma)N_\Gamma/N_\Gamma}{\Gamma/N_\Gamma}$ as an ideal of $k(\Gamma/N_\Gamma)$. Hence by the commutativity of  \eqref{diagram2},  
	 $$ \tilde \phi\left(\augNor{\mho_t(G)N_G/N_G}{G/N_G}\right)=\augNor{\mho_t(H)N_H/N_H}{H/N_H}.$$
	 Now the result follows from \Cref{Triviality} and the commutativity of \eqref{diagram1} as in the previous item.

	  \eqref{mainlemma2} Using the decomposition of the center \eqref{CenterDecomposition}, we get that 
	  \begin{align*}
	  	\ZZ(k\Gamma)\cap \ker \lambda_\Gamma&= k\ZZ(\Gamma)\cap \ker \lambda_\Gamma+ [k\Gamma,k\Gamma]\cap \ZZ(k\Gamma)\cap \ker \lambda_\Gamma\\
	  	&=\augNor{\Omega_t(\ZZ(\Gamma))}{\ZZ(\Gamma)} +   [k\Gamma,k\Gamma]\cap \ZZ(k\Gamma)\cap \ker \lambda_\Gamma ,
	  \end{align*}
  where the second equality is due to \Cref{lemma:abelianppower}. 
  As $[k\Gamma,k\Gamma]\subseteq \augNor{\Gamma'}{\Gamma}\subseteq \augNor{N_\Gamma}{\Gamma}$, it follows that the ideal of $k\Gamma$ generated by $\ZZ(k\Gamma)\cap \ker\lambda_\Gamma$ and $\augNor{N_\Gamma}{\Gamma}$ is exactly 
  $\augNor{\Omega_t(\ZZ(\Gamma))N_\Gamma}{\Gamma}.$ Thus $\phi(\ZZ(kG)\cap \ker \lambda_G) =\ZZ(kH)\cap \ker \lambda_H$ and $\phi(\augNor{N_G}{G})=\augNor{N_H}{H}$ imply the result.
  
\end{proof}

\begin{examples} \rm We illustrate some uses of the Transfer Lemma and relate it to some (recent and classical) known results, as a  unified way to approach them. Let $k$ be a field of characteristic $p$, $G$ and $H$ finite $p$-groups, and  $\phi:kG\to kH$ be an isomorphism preserving the augmentation, and fix $t\geq 1$.
	\begin{enumerate} 
		\item   Taking $t=1$ and $N_G=G'$, \Cref{mainlemma}.\eqref{mainlemma3} yields that
		\begin{equation}
			\label{CenterDerivedCanonical}\phi(\augNor{\ZZ(G)G'}{G}=\augNor{\ZZ(H)H'}{H}
		\end{equation}   This was already proven in \cite{BK07} in order to show that if $G$ is of nilpotency class $2$ and $KG\cong kH$, then $H$ also has nilpotency class $2$.
		\item Taking   $N_\Gamma=\Gamma'$ and $L_\Gamma=\Gamma$,  \Cref{mainlemma}\eqref{mainlemma3} yields that 
		\begin{equation}\label{Salim1}
			\phi\left(\augNor{\mho_t(G)G'}{G}\right)=\augNor{\mho_t(H)H'}{H},
		\end{equation}
	which is \cite[Theorem 3.2]{SalimTesis}. In particular for $t=1$ this yields
	$ \phi\left(\augNor{\Phi(G)}{G}\right) =\augNor{\Phi(H)}{H} $. For $N_\Gamma=G'$, \Cref{mainlemma}\eqref{mainlemma1}  becomes
	$
		 \phi(\augNor{\Omega_t(G:G')}{G})=\augNor{\Omega_t(H:H')}{H}$,
which is  \cite[Lemma 3.5]{SalimTesis}.
		\item  We also observe   that  \cite[Theorem B]{MSS21}   follow  as   special   cases  of the previous lemmas.  Indeed,  applying \Cref{JenningsNor}\eqref{JenningsNor3} to \eqref{Salim1} we obtain  that
		 $ \ZZ(G)\cap \mho_t(G)G'\cong \ZZ(H)\cap \mho_t(H)H'$ and $ \ZZ(G)\mho_t(G)G'/\mho_t(G)G'\cong \ZZ(H)\mho_t(H)H'/\mho_t(H)H'$.  Furthermore, for $N_\Gamma=\Gamma'$, \Cref{mainlemma}\eqref{mainlemma2} yields that 
		\begin{equation} \label{eq:example2}
			\phi\left(\augNor{\Omega_t(\ZZ(G))G'}{G}\right)=\augNor{\Omega_t(\ZZ(H))H'}{H},
		\end{equation}
	and hence applying \Cref{JenningsNor}\eqref{JenningsNor2} we get that 
			$	G/\Omega_t(\ZZ(G))G'\cong H/\Omega_t(\ZZ(H))H'$ and   $   \Omega_t(\ZZ(G))G'/G'\cong \Omega_t(\ZZ(H))H'/H'$,
the two remaining invariants.
	 Some other invariants resembling the previous ones but not appearing in the mentioned theorem also follow  readily. For example,   for   $L_\Gamma=\ZZ( \Gamma)\Gamma'$ and $N_\Gamma=\Gamma'$, \Cref{mainlemma}\eqref{mainlemma3} yields $
			\phi(\augNor{\mho_t(\ZZ(G))G'}{G})=\augNor{\mho_t(\ZZ(H))H'}{H} 
		$. So that applying \Cref{JenningsNor}\eqref{JenningsNor2} we derive that $$
	G/\mho_t(\ZZ(G))G'\cong H/\mho_t(\ZZ(H))H'\qand  \mho_t(\ZZ(G))G'/G'\cong \mho_t(\ZZ(H))H'/H'.$$ 
These two last invariants, as far as we know, were never considered before.
		
		\item For $N_\Gamma=\ZZ(G)G'$, \Cref{mainlemma}\eqref{mainlemma1} and \eqref{CenterDerivedCanonical}  yield  that 
		$ \phi\left(\augNor{\Omega_t(G:\ZZ(G)G')}{G}\right) =\augNor{\Omega_t(H:\ZZ(H)H')}{H}$. 
		This leads, with the help of \Cref{JenningsNor}, to a number of invariants of the group algebra, which, to the best of our knowledge, were used for the first time   in \cite{GLdRS2022}. There it is shown   (see \cite[Lemma 4.1]{GLdRS2022}) that for $p$-groups with cyclic derived subgroup and $p>2$ there exists an integer $t$ depending only on $kG$ such that 
		$ C_G(G')=\Omega_t( G:\ZZ(G)G') $,  
		so over this family of groups the subgroup assignation  $G\mapsto C_G(G')$ is $k$-canonical. 
	\end{enumerate}
	 
 \end{examples}

\section{Abelian direct factors}\label{section:proofs}
 This section is devoted to prove our main results, and, after a short introduction, it is divided in two parts:  the first one is focused on \Cref{mainprop}, while the second one  contains the proof of \Cref{maintheorem}.

 Let $G$ be a finite $p$-group. A  \emph{homocyclic decomposition} $\mathfrak d$ of $G$ is an internal direct product decomposition  \begin{equation}\label{decomposition}
	\mathfrak d:\qquad \qquad G= U_1\oplus \dots \oplus U_l\oplus A_1\oplus \dots \oplus A_k
\end{equation}  
where $U_i,A_i$ are subgroups of $G$, $U_i$ is non-abelian and indecomposable, and $A_i$ is homocyclic of exponent $p^{i}$ and rank $r_i$, i.e., $A_i\cong(C_{p^{i}})^{r_i}$. Here we allow $r_i=0$, in which case $A_i=1$.
Such a decomposition always exists by the Krull-Remak-Schmidt theorem.   Given a decomposition \eqref{decomposition}, denote  $H_i^{\mathfrak d}(G)=A_i $.   Observe that the isomorphism type of $H_i^{\mathfrak d}(G)$ does not depend on $\mathfrak d$,  also by the Krull-Remak-Schmidt theorem. Sometimes we are only interested in the isomorphism type of $H_i^{\mathfrak d}(G)$, and in such case we  drop $\mathfrak d$ from the notation.   We say that $H_i (G)$ is the \emph{homocyclic component of $G$ of exponent $p^i$}. Moreover, if we can express $G$ as an internal direct product $G=S\oplus T$, where $T$ is homocyclic of exponent $p^i$, then $G$ has a homocyclic decomposition  $\mathfrak d$ satisfying $T\subseteq H_i^{\mathfrak d}(G)$. In this case we say that $\mathfrak d$ \emph{extends} the decomposition $G=S\oplus T$.

With the   notation above, we also denote $\text{\rm Ab}^{\mathfrak d}(G)=A_1\oplus \dots \oplus A_k $ and $\text{\rm NAb}^{\mathfrak d}(G)=U_1\oplus\dots \oplus U_l$. The same considerations using the Krull-Remak-Schmidt theorem yield that the isomorphism types of these subgroups do  not depend on $\mathfrak d$. Hence we can drop it from the notation and write
$$G\cong \text{\rm Ab}(G)\oplus\text{\rm NAb}(G), $$
so this  notation agrees with the one used in the introduction. 
 
 \vspace*{0.2cm}

For the rest of the section, let $t$ be a positive integer. The map
$$\lambda_G^{t-1}:\frac{\Omega_t(\ZZ(G)) \Phi(G)}{  \Phi(G)}\to \frac{\mho_{t-1}(G)G'}{  \mho_t(G)G'}, \quad x\Phi(G) \mapsto x^{p^{t-1}}\mho_t(G)G',$$  
which is a  homomorphism of elementary abelian $p$-groups, will play an important role in the proof of \Cref{maintheorem} as a tool to detect homocyclic components. We continue by listing some elementary properties of the homocyclic components.

 \begin{lemma}\label{list}
 	Let $\mathfrak d$ be a homocyclic decomposition of $G$. Then \begin{enumerate}
 		\item \label{l1} $H_t^{\mathfrak d}(G)\subseteq \Omega_t(\ZZ(G))$. 
 		\item \label{l2}$d(H_t^{\mathfrak d}(G))=d(H_t^{\mathfrak d}(G)\Phi(G)/\Phi(G))$.
 		\item \label{l3-}If $U$ is a subgroup of $G$ such that $G=U\oplus H_t^{\mathfrak d}(G)$, then $\mho_t(G)G'=\mho_t(U)U'$.
 		\item \label{l3} $H_t^{\mathfrak d}(G)\cap \mho_t(G)G'=1$.
 		\item\label{l4} The map $\lambda_G^{t-1}$ restricted to 
 		$ H_t^{\mathfrak d}(G)\Phi(G)/\Phi(G)  $ is injective. 
 	\end{enumerate} 
 \end{lemma}
 \begin{proof}
 	\eqref{l1} and \eqref{l2} are obvious.  Let $U$ be a subgroup of $G$ such that $G=U\oplus H_t^{\mathfrak d}(G)$. As $H_t^{\mathfrak d}(G)$ is abelian, it follows that $G'=U'$; similarly $\mho_t(G)=\mho_t(U)$ because $H_t^{\mathfrak d}(G)$ has exponent $p^t$, and \eqref{l3-} follows. In particular $\mho_t(G)G'\cap H_t^{\mathfrak d}(G)\subseteq U\cap H_t^{\mathfrak d}(G)=1$. This proves  \eqref{l3}. Finally, observe that, in general, for a  homocyclic group $H$ of exponent $p^t$, every element in $H\setminus \Phi(H)$ has order $p^t$. Therefore if $x\in H_t^{\mathfrak d}(G)$ is  such that $x^{p^{t-1}}\in \mho_t(G)G'$, then by \eqref{l2} we have that $x^{p^{t-1}}=1$, and hence $x\in \Phi(H_t^{\mathfrak d}(G))\subseteq \Phi(G)$. 
 \end{proof}
 
 \subsection{Proof of \Cref{mainprop}}

 \begin{lemma}\label{lemma:directfactorNuevo}
 		Let $T$ be a  homocyclic subgroup of $\Omega_t(\ZZ(G))$ with exponent $p^t$ such that $d(T)=d(T\Phi(G)/\Phi(G))$ and  $T\cap  \mho_t(G)G'=1$. Then $G=S\oplus T$, for some subgroup $S$ of $G$.
 \end{lemma}
\begin{proof}To simplify the notation, we set $N=\mho_t(G)G'$. 
	The proof will rely on the following 
	
	\vspace*{0.1cm}
	
	\noindent \textbf{Claim}. Let  $\{x_1,\dots, x_n,y\}$ be a Burnside basis of $G$ and assume that $y\in \Omega_t(\ZZ(G))$  has order $p^t$ and verifies $\GEN{y}\cap N=1$. Then 
	 $$G=\GEN{x_1'  ,\dots, x_n'}\oplus \GEN{y}, $$ 
	where $x_i'= x_i^{w_i} y^{e_i}$, for some integers $w_i$ and $e_i$ with $p\nmid w_i$, for each $i$.

	\vspace*{0.1cm}
	
\noindent	\textit{Proof of the claim}:  We set $S_0=1$ and construct   $x_j'=x_j^{w_j}y^{e_j}$ recursively such that $S_jN\cap \GEN{y}=1$, where $S_j=\GEN{x_1',\dots,x_j'}$.
	Suppose that $j\ge 0$ and $x'_1,\dots,x'_j$ have been already constructed with $S_jN\cap \GEN{y}=1$.
	If $\GEN{S_j,x_{j+1} }N \cap \GEN{y}=1$, then set $x_{j+1}'=x_{j+1}$ (i.e., $e_{j+1}=0$ and $w_{j+1}=1$), $S_{j+1}=\GEN{S_j,x_{j+1}}$, and continue. Assume on the contrary that $\GEN{S_j,x_{j+1}}N\cap \GEN{y}=\GEN{y^{p^e}}\neq 1$. In particular $e<t$. Then taking quotients modulo $S_jN$ one has that $\GEN{x_{j+1}S_jN}\cap \GEN{yS_jN}=\GEN{y^{p^e}S_jN}\neq 1$, since if $y^{p^e}\in S_jN$ then $y^{p^e}=1$ by the hypothesis $S_jN\cap \GEN{y}=1$. Thus there is an integer $wp^s$ with $p\nmid w$ such that $x_{j+1}^{wp^s} y^{-p^e}\in S_jN$. If $s\geq t$, then $y^{p^e}\in S_jN$, a contradiction. Thus $s<t$. If $ s>e$ then  $x_{j+1}^{wp^t}y^{-p^{e+t-s}}\in S_jN$, so $1\neq y^{p^{e+t-s}}\in S_jN$, a contradiction. Thus $s\leq e$. Then set $x_{j+1}'=x_{j+1}^w y^{p^{e-s}}$ and $S_{j+1}=\GEN{S_j,x_{j+1}'}$. The previous argument shows that $S_{j+1}N\cap \GEN{y}=1$. This finishes the recursive construction of the $x_j'$'s. Then $S_n\cap \GEN{y}=1$ and hence $G=\GEN{x_1,\dots, x_n,y}=\GEN{x_1',\dots, x_n',y}=S_n\oplus \GEN{y}$.  This finishes the proof of the claim. 

	\vspace*{0.1cm}

Now we are ready to prove the lemma. Let $\{y_1,\dots, y_r\}$ be a Burnside basis of $T$. Clearly each $y_i$ has order $p^t$, as $T$ is homocyclic with exponent $p^t$. Since $d(T\Phi(G)/\Phi(G))=d(T)=r$,  we can extend it to a Burnside basis $\{x_1,\dots, x_n,y_1,\dots, y_r\}$ of $G$. Then we can apply the claim $r$ times to obtain a decomposition 	$$G=\left( \left(  \left( \GEN{x_1',\dots, x_n'} \oplus \GEN{y_1'}\right)\oplus\GEN{y_2'}\right) \dots \right) \oplus \GEN{y_r}= \GEN{x_1',\dots, x_n'}\oplus \GEN{y_1',\dots,  y_{r-1}',y_r},$$
where, for $1\leq i<r$,
$$y_i'= y_i^{w_i} y_{i+1}^{e_{i,i+1}}\dots y_{r }^{e_{i,r }}   $$
for suitable integers $e_{i,j}  $ and $w_i$, with $p\nmid w_i$. This implies that $\GEN{y_1',\dots, y_{r-1}',y_r}=\GEN{y_1,\dots, y_r}=T$, so the lemma follows taking $S=\GEN{x_1',\dots, x_n'}$.
\end{proof}

\begin{lemma}\label{lemma:HPiOmega}One has
	$$H_t(G)\cong H_t\left(\frac{  \Omega_t(\ZZ(G))\mho_{t}(G)G' }{\mho_{t}(G)G'} \right) . $$
\end{lemma}
 \begin{proof} Fix a  homocyclic decomposition $\mathfrak d$ of $G$. Let $\pi:G\to G/\mho_{t}(G)G'$ be the natural projection. By \Cref{list}\eqref{l3} the map $\pi$ restricted to $H_t^{\mathfrak d}(G)$ is injective. Let $V$ be a subgroup of $G$ such that $G=V\oplus H_t^{\mathfrak d}(G)$, and write $U=V\cap \Omega_t(\ZZ(G))$.   Then  $\Omega_t(\ZZ(G))=U\oplus H_t^{\mathfrak d}(G)$. Moreover, $\pi(\Omega_t(\ZZ(G)))=\pi(U)\oplus \pi(H_t^\mathfrak d(G))$. Indeed, it suffices to prove that $ \pi(U)\cap \pi(H_t^\mathfrak d(G))=1$, and to show this we observe that  	$\mho_t(G)G'=\mho_t(V)V'\subseteq V$ by \Cref{list}\eqref{l3-}, and hence $U\mho_t(G)G'\cap (H_t^{\mathfrak d}(G)\mho_t(G)G')\subseteq V\mho_t(G)G'\cap (H_t^{\mathfrak d}(G)\mho_t(G)G')= V\cap (H_t^{\mathfrak d}(G)\mho_t(G)G')=(V\cap H_t^{\mathfrak d}(G))  \mho_t(G)G'=\mho_t(G)G'$.
 	
 	  We can extend the decomposition $ \pi(\Omega_t(\ZZ(G)))=\pi(U)\oplus \pi(H_t^\mathfrak d(G))$ to a  homocyclic decomposition $\bar{\mathfrak d}$ of $\pi(\Omega_t(\ZZ(G)))$. Then    $H_t^{\mathfrak d}(G)\cong \pi(H_t^{\mathfrak d}(G))\subseteq H_t^{\bar {\mathfrak d } }(\pi(\Omega_t(\ZZ(G))) )$. Assume by contradiction that the inclusion is strict. Then there is a non-trivial subgroup $A\subseteq H_t^{\bar{\mathfrak d}}(\pi(\Omega_t(\ZZ(G))) )$ such that $H_t^{\bar{\mathfrak d}}(\pi(\Omega_t(\ZZ(G))) )= A\oplus \pi( H_t^{\mathfrak d}(G))$, so that $A$ must be homocyclic with exponent $p^t$. Take an element $x\in A$ with order $p^t$. Then there is some element $y\in \Omega_t(\ZZ(G))$ such that $\pi(y)=x$. The order of $y$ is also $p^t$. If $y\in \Phi(G)$ then   $y^{p^{t-1}}\in \mho_{t}(G)G'$, so that $x^{p^{t-1}}=\pi(y^{p^{t-1}})=1$, a contradiction. Moreover   $\GEN{y}\cap H_t^{\mathfrak d}(G)=1$, as $\pi(\GEN{y}\cap H_t^{\mathfrak d}(G))  =\GEN{x}\cap \pi(H_t^{\mathfrak d}(G))=1$.  Now observe that $\pi(\GEN{y}\oplus H_t^{\mathfrak d}(G))=\GEN{x}\oplus \pi(H_t^{\mathfrak d}(G))$, so that $\pi$ is also injective over $\GEN{y}\oplus H_t^{\mathfrak d}(G)$, i.e., $(\GEN{y}\oplus H_t^{\mathfrak d}(G) )\cap \mho_t(G)G'=1$. Finally observe that if $y\in \Phi(G)H_t^{\mathfrak d}(G)$ then $y^{t-1}\in \mho_t(G)G'H_t^{\mathfrak d}(G)$, so that $1\neq x^{p^{t-1}}=\pi(y^{t-1})\in \GEN{x}\cap \pi(H_t^{\mathfrak d}(G))=1$, a contradiction. Hence the rank of $(\GEN{y}\oplus H_t^{\mathfrak d}(G))\Phi(G)/\Phi(G)$ is the rank of $\GEN{y}\oplus H_t^{\mathfrak d}(G)$. Therefore by \Cref{lemma:directfactorNuevo} there exists a subgroup $S$ of $G$ such that
 $ G=S\oplus \GEN{y}\oplus H_t^{\mathfrak d}(G)  $.
 But now we can extend this decomposition to a homocyclic decomposition $\tilde {\mathfrak d}$ of $G$ verifying that $  \GEN{y}\oplus H_t^{\mathfrak d}(G)\subseteq H_t^{\tilde {\mathfrak d}}(G)\cong H_t^{\mathfrak d}(G)$, a contradiction.

 \end{proof}

\begin{lemma}\label{lemma:AbGDecomposition} One has
	$$\text{Ab}(G)\cong \bigoplus_{i\geq 1}  H_i\left( \frac{\Omega_i(\ZZ(G))\mho_i(G)G'}{  \mho_i(G)G' }\right) .$$
\end{lemma}
\begin{proof} 
	Since $\text{Ab}(G)\cong  \bigoplus_{i\geq 1} H_i(G)$,  \Cref{lemma:HPiOmega} yields  the result. 
\end{proof}

The following is the same as \Cref{mainprop}.
\begin{proposition}\label{lemma:abelianfactor}
	Let $G$ and $H$ finite $p$-groups such that $kG\cong kH$. Then $\text{\rm Ab}(G)\cong \text{\rm  Ab}(H)$. 
\end{proposition}
\begin{proof} Let $\phi:kG \to kH$ be an isomorphism preserving the augmentation.   By \eqref{Salim1} we can apply \Cref{mainlemma}\eqref{mainlemma2}    for $N_\Gamma=\mho_t(\Gamma)\Gamma'$and derive that $
	\phi\left(\augNor{\Omega_t(\ZZ(G))\mho_t(G)G'}{G}\right)=\augNor{\Omega_t(\ZZ(H))\mho_t(H)H'}{H}$. By the same reason, we can apply \Cref{JenningsNor}\eqref{JenningsNor2-1/2} with $N_\Gamma=\mho_t(\Gamma)\Gamma'$ and $L_\Gamma=\Omega_t(\ZZ(\Gamma))\mho_t(\Gamma)\Gamma'$ to deduce that 
	$$\frac{\Omega_t(\ZZ(G))\mho_t(G)G'}{ \mho_t(G)G' }\cong \frac{\Omega_t(\ZZ(H)) \mho_t(H)H'}{\mho_t(H)H'}$$
	for each $t\geq 1$. In particular
	$$ H_t\left(\frac{\Omega_t(\ZZ(G))\mho_t(G)G'}{ \mho_t(G)G'}\right) \cong H_t\left( \frac{\Omega_t(\ZZ(H)) \mho_t(H)H'}{\mho_t(H)H'}\right).$$
	Thus the result follows from \Cref{lemma:AbGDecomposition}. 
\end{proof}

In particular the previous lemma  is equivalent to 
$ H_t(G)\cong H_t(H)  $ for each $t$, provided that $kG\cong kH$. Hence   $|H_t(G)|=|H_t(H)|$.

\subsection{Proof of \Cref{maintheorem}} We will need another pair  of lemmas about homocyclic components, as well as to recover some ideas from \cite{MSS21}.

\begin{lemma}\label{lemma:injectivemap}
	Let $T$ be a subgroup of $\Omega_t(\ZZ(G))$ such that  $d(T)=d(T\Phi(G)/\Phi(G)) $ and such that the restriction of $\lambda_G^{t-1}$ to $T\Phi(G)/\Phi(G)$ is injective. Then  $T$ is homocyclic of exponent $p^t$ and $G=S\oplus T$ for some subgroup $S$ of $G$.
\end{lemma} 
\begin{proof}Observe that   $d(T)=d(T\Phi(G)/\Phi(G))$ implies that $\Phi(G)\cap T=\Phi(T)$.  As $T\subseteq  \Omega_t(\ZZ(G))$, to show that it is homocyclic of exponent $p^t$ it suffices to prove that every element in $T\setminus \Phi(T)=T\setminus \Phi(G)$ has order greater than $p^{t-1}$, and that is a direct consequence of the injectivity of $\lambda_G^{t-1}$.

	We claim that $T\cap \mho_t(G)G'=1$. Indeed, if $1\neq x\in T\cap \mho_t(G)G'$ then there exist  some $y\in G$ and $z\in G'$ such that $x=y^{p^t}z$. As $x\in T\subseteq  \Omega_t( \ZZ(G))$, there is some integer $e$ and some element $g\in T\setminus \Phi(T)=T\setminus \Phi(G)$ such that $g^{p^e}=x$, with $e<t$. Therefore $g^{p^{t-1}}=x^{p^{t-1-e}}=(y^{p^t}z)^{p^{t-1-e}}= y^{p^{2t-e-1}}z'$ for some $z'\in  G'$, that is, $g^{p^{t-1}}\in \mho_t(G)G'$, thus $1\neq g \Phi(G)$ belongs to the kernel of $\lambda_G^{t-1}$, a contradiction.  So the claim follows, and we can apply  \Cref{lemma:directfactorNuevo} to derive the result.
\end{proof}

\begin{lemma}\label{lemma:complement}
Let $T$ be a subgroup of $\Omega_t(\ZZ(G))$ such that  $d(T)=d(T\Phi(G)/\Phi(G)) $. Then the following conditions are equivalent:
\begin{enumerate}
	\item \label{lemma:complement1}$\Omega_t(\ZZ(G))\Phi(G)/\Phi(G)$ admits the following direct product decomposition 
	$$\frac{\Omega_t(\ZZ(G))\Phi(G)}{\Phi(G)}=\frac{T\Phi(G)}{\Phi(G)}\oplus \ker \lambda_G^{t-1}  .$$ 
	\item There exists a homocyclic decomposition $\mathfrak d$ of $G$ such that $T=H_t^{\mathfrak d}(G)$.
\end{enumerate}
	
\end{lemma}
\begin{proof}  Assume \eqref{lemma:complement1}. This implies that the restriction of $\lambda_G^{t-1}$ to $T\Phi(G)/\Phi(G)$ is injective. Thus by \Cref{lemma:injectivemap} we have that $G=S\oplus T$ for some subgroup $S$ of $G$, and $T$ is homocyclic of exponent $p^t$. Hence there is a  homocyclic decomposition $\mathfrak d$ of $G$ such that $T\subseteq H_t^{\mathfrak d}(G) $. Suppose by contradiction that the inclusion is strict. Then, observing that the elementary abelian group $\Omega_t(\ZZ(G))\Phi(G)/\Phi(G)$ can be seen as a vector space over $\F_p$, by counting dimensions it is clear that  $(H_t^{\mathfrak d}(G)\Phi(G)/\Phi(G)) \cap \ker(\lambda_G^{t-1}) \neq 1$. Hence $\lambda_G^{t-1}$ is not injective over $H_t^{\mathfrak d}(G)$, contradicting \Cref{list}\eqref{l4}.

	Conversely, assume that $T=H_t^{\mathfrak d}(G)$ for some homocyclic decomposition $\mathfrak d$ of $G$.  Since $\lambda_G^{t-1}$ restricted to $T\Phi(G)/\Phi(G)$ is injective by \Cref{list}\eqref{l4}, we have that $(T\Phi(G)/\Phi(G))\cap \ker  \lambda_G^{t-1} =1$. Hence $(T\Phi(G)/\Phi(G))\oplus \ker\lambda_G^{t-1}\subseteq \Omega_t(\ZZ(G))\Phi(G)/\Phi(G)$. It suffices to show that the inclusion is in fact an equality, and to see this we count dimension as $\F_p$-spaces. Take a direct complement $U$ of $\ker \lambda_G^{t-1}$ in $\Omega_t(\ZZ(G))\Phi(G)/\Phi(G)$. Let $x_1,\dots, x_r$ be a basis of $U$. Then we can take elements $y_1,\dots, y_r\in \Omega_t(\ZZ(G))$ such that $y_i\Phi(G)=x_i$. Set $S=\GEN{y_1,\dots, y_r}$. Then $S\subseteq \Omega_t(\ZZ(G))$, $d(S)=r=d(S\Phi(G)/\Phi(G))$ and $S\Phi(G)/\Phi(G)\oplus \ker \lambda_G^{t-1}=\Omega_t(\ZZ(G))\Phi(G)/\Phi(G)$. Thus by the previous paragraph there is a homocyclic decomposition $\bar{\mathfrak d}$ of $G$ such that $S=H_t^{\bar{\mathfrak d}}(G)$. Thus 
	$$ \dim(T\Phi(G)/\Phi(G))\leq \dim(\Omega_t(\ZZ(G))\Phi(G) /\Phi(G))-\dim(\ker \lambda_G^{t-1})= \dim(H_t^{ \bar{\mathfrak d}}(G)\Phi(G)/\Phi(G)).$$
By  \Cref{list}\eqref{l2}  the first and the last terms in the previous expression equal $d(H_t^{\mathfrak d}(G)) =d(H_t(G))$, so the inequality is in fact an equality and the result follows.  
	 
\end{proof}

%Now we replicate the argument of \cite{MSS21} to obtain a generalization of \cite[Theorem A]{MSS21}, extending this result to direct products by abelian groups, without restriction on the exponent. However, our proof apparently works only for the field with $p$ element. 

\begin{lemma} \cite[Lemma 4.2]{MSS21} \label{lemma:semidirect1}
	Let $N$ and $L$ be subgroups of $G$ such that $G$ is the direct product
	$G =N \times L$.   Then for each positive integer $n$
	the following equality holds:
  	$$\aug{G}^n= \aug{N}\aug{G}^{n-1}\oplus \aug{L}^n.$$ 
\end{lemma}

The following lemma extracts the idea, with essentially the same proof, of \cite[Lemma 4.5]{MSS21}.
\begin{lemma}   \label{lemma:semidirect2}
		Let $N$ and $L$ be subgroups of $G$ such that $G$ is the  direct product
	$G =N \times L$. Let $J$ be a proper ideal of $kG $ such that:
	\begin{enumerate}
		\item \label{lemma:semidirect2.1}$\text{\rm codim}_{kG}(J)=|L|$. 
		\item\label{lemma:semidirect2.2} $J+\aug{G}^2=\aug{N}kG+\aug{G}^2$.
	\end{enumerate}
Then $kG=J\oplus kL$.
\end{lemma}
\begin{proof} Observe that \eqref{lemma:semidirect2.2} easily implies that  	for each $n\geq 1$
	\begin{equation}\label{eq:semidirect}
		J\aug{G}^{n-1}+ \aug{G}^{n+1}=\aug{N}\aug{G}^{n-1}+\aug{G}^{n+1} .
	\end{equation}  
	We claim that 
	$$ \aug{G}^n=J\aug{G}^{n-1}+\aug{L}^n $$
	for each $n\geq 1$.  Let $c$ the smallest integer such that $\aug{G}^c=0$. If $n\geq c$ then both terms in the equality are zero since $J\subseteq \aug{G}$, thus it holds trivially. Assume by reverse induction that $\aug{G}^{n+1}=J\aug{G}^n+\aug{L}^{n+1}$. Then
		\begin{eqnarray*} 
	     \aug{G}^n&=&\aug{G}^n+\aug{G}^{n+1} \qquad   \\
		&=&\aug{N}\aug{G}^{n-1} + \aug{L}^n+\aug{G}^{n+1}  \hspace*{2cm} 	\textrm{(by \Cref{lemma:semidirect1})} 	 \\
		  &=&J\aug{G}^{n-1}+\aug{L}^n+\aug{G}^{n+1}  \hspace*{2.5cm}  \textrm{(by \eqref{eq:semidirect})}    \\
			&=& J\aug{G}^{n-1}+\aug{L}^n+   J\aug{G}^n+\aug{L}^{n+1}  \hspace*{1cm}  \textrm{(by the induction hypothesis)}  \\
	 	&=&J\aug{G}^{n-1}+\aug{L}^n.
	\end{eqnarray*} 
Thus the claim follows. In particular, for $n=1$ we have that $\aug{G}=J+\aug{L}$. So that $kG=J+kL$, and \eqref{lemma:semidirect2.1} guarantees that the sum is direct. 
\end{proof}

\begin{lemma}\label{lemma:normal+jennings}
	If  $N$ is a normal subgroup of $G$ then 
	$ \left(1+\augNor{N}{G}+\aug{G}^n\right)\cap G= \M_n(G)N$.
\end{lemma}
\begin{proof}
	The right-to-left inclusion is trivial, so we prove the converse one. Let $\pi_N :kG\to k(G/N)$ be the natural projection and assume that $g\in \left(1+\augNor{N}{G}+\aug{G}^n\right)\cap G$. Then $\pi_N(G)\in \left(1+\aug{G/N}^n\right)\cap (G/N)= \M_n(G/N)\subseteq\M_n(G)N/N$, so that $g\in \M_n(G)N$.
\end{proof}

%Fix an integer $t\geq 1$. 
We will use that the following map
$$
	\Lambda_G^{t-1}:\frac{\aug{G}}{\aug{G}^2} \to \frac{\aug{G}^{p^{t-1}}}{\aug{G}^{p^{t-1}+1}} , \quad 
	 x+\aug{G}^2 \mapsto x^{p^{t-1}}+\aug{G}^{p^{t-1} +1},
$$
  already introduced in \cite{Passman1965p4}, is well defined. 
 Moreover there are injective maps $$
	\psi_n :\frac{\M_n(G)}{\M_{n+1}(G)} \to \frac{\aug{G}^n}{\aug{G}^{n+1}}, \quad  
	 x\M_{n+1}(G) \mapsto x-1 +\aug{G}^{n+1} . 
$$  \underline{From now on we assume that $k=\F_p$.} Then $\psi_1$ is an isomorphism (see, for example, \cite[Proposition III.1.15]{Seh78}).   Let $N$ be a normal subgroup of $G$. Therefore $\psi_1(N\Phi(G)/\Phi(G))=(\augNor{N}{G}+\aug{G}^2)/\aug{G}^2$.  Furthermore,  from  \Cref{lemma:normal+jennings} it follows easily  that 
$$
	\psi_n^N :\frac{\M_n(G)N}{\M_{n+1}(G)N} \to \frac{\aug{G}^n+ \augNor{N}{G}}{\aug{G}^{n+1}+ \augNor{N}{G}}, \quad
	x\M_{n+1}(G)N \mapsto x-1 +\aug{G}^{n+1}+\augNor{N}{G}  
$$
is also injective. 
Therefore we have a commutative diagram 
\begin{equation*}
	\xymatrix{
		\frac{G}{\Phi(G) } \ar[d]_-{\psi_1} \ar[r]^-{\tilde \lambda_G^{t-1}} & \frac{\M_{p^{t-1}}(G)N}{\M_{p^{t-1}+1}(G)N} \ar[d]^-{\psi_{p^{t-1}}^N} \\
			\frac{\aug{G}}{\aug{G}^2} \ar[r]_-{\Lambda_G^{t-1}} & \frac{\aug{G}^{p^{t-1}}+\augNor{N}{G}}{\aug{G}^{p^{t-1}+1}+ \augNor{N}{G}}
	}
\end{equation*}
where $\tilde \lambda_G^{t-1}$ is given by the $p^{t-1}$-power map in the group, $x\Phi(G)\mapsto x^{p^{t-1}}\M_{p^{t-1}+1}(G)N$. Observe that the image of $\tilde\lambda_G^{t-1} $ is contained, by definition, in $\mho_{t-1}(G)$. Moreover, restricting this diagram to $\Omega_t(\ZZ(G))\Phi(G)/\Phi(G)$ we have that 
\begin{equation*}
	\xymatrix{
		\frac{\Omega_t(\ZZ(G))\Phi(G)}{\Phi(G) } \ar[d]_-{\psi_1} \ar[r]^-{\tilde \lambda_G^{t-1}} & \frac{\mho_{t-1}(G)N}{\M_{p^{t-1}+1}(G)N} \ar[d]^-{\psi_{p^{t-1}}^N} \\
		\frac{\augNor{\Omega_t(\ZZ(G))G'}{G}+\aug{G}^2}{\aug{G}^2} \ar[r]_-{\Lambda_G^{t-1}} & \frac{\aug{G}^{p^{t-1}}+\augNor{N}{G}}{\aug{G}^{p^{t-1}+1}+ \augNor{N}{G}}
	}
\end{equation*}
commutes, where $\psi_1$ is still an isomorphism. Now set $N=\mho_t(G)G'$. As $\mho_t(G)\subseteq \mho_{t-1}(G)$ and $\M_{p^{t-1}+1}(G)\subseteq \mho_t(G)G'$, we have that $\tilde \lambda_G^{t-1}=\lambda_G^{t-1}$ and the previous diagram becomes:
\begin{equation}\label{diagram:prelema}\begin{array}{c} 
	\xymatrix{
		\frac{\Omega_t(\ZZ(G))\Phi(G)}{\Phi(G) } \ar[d]_-{\psi_1} \ar[r]^-{\lambda_G^{t-1}} & \frac{\mho_{t-1}(G)G'}{\mho_t(G)G'} \ar[d]^-{\psi_{p^{t-1}}^{\mho_t(G)G'}} \\
		\frac{\augNor{\Omega_t(\ZZ(G))G'}{G}+\aug{G}^2}{\aug{G}^2} \ar[r]_-{\Lambda_G^{t-1}} & \frac{\aug{G}^{p^{t-1}}+\augNor{\mho_t(G)G'}{G}}{\aug{G}^{p^{t-1}+1}+ \augNor{\mho_t(G)G'}{G}}.
	}
\end{array}
\end{equation} 
Observe that by the election of $N$ the diagram can be seen as constructed from the commutative group algebra $k(G/G')$, so that  both $\lambda_G^{t-1}$ and $\Lambda_G^{t-1}$ are homomorphisms of elementary abelian groups/vector spaces over $k$. %Moreover $\lambda_G^{t-1}$ agrees with the map defined at the beginning of the  section. 
\begin{lemma}\label{subspaceV}
	Let $V_G $ be a subspace of $\aug{G}$ containing $\aug{G}^2$. With the notation above, the following conditions are equivalent:
	
	\begin{enumerate}
		\item There is a direct sum decomposition
		\begin{equation}
		\label{eq:directsum}\frac{V_G }{\aug{G}^2}\oplus \ker \Lambda_G^{t-1}=	\frac{\augNor{\Omega_t(\ZZ(G))G'}{G}+\aug{G}^2}{\aug{G}^2} .
		\end{equation}  
		 \item There exists a  homocyclic decomposition $\mathfrak d$ of $G$ such that 
		 $$  V_G = \augNor{H_t^{\mathfrak d}(G)}{G}+\aug{G}^2.$$
	\end{enumerate}

\end{lemma}
\begin{proof}   Since $\psi_1$ is an isomorphism, $\psi_{p^{t-1}}^{\mho_t(G)G'}$ is injective,  and the diagram \eqref{diagram:prelema} is commutative, \eqref{eq:directsum} is equivalent to
	\begin{equation}\label{eq:directproduct}
	\psi_1^{-1}\left(\frac{V_G }{\aug{G}^2} \right)\oplus \ker\lambda_G^{t-1}=\frac{\Omega_t(\ZZ(G))\Phi(G)}{\Phi(G)}.
	\end{equation} 
	If $V_G =\augNor{H_t^{\mathfrak d}(G)}{G}+\aug{G}^2$ for some decomposition $\mathfrak d$ of $G$ then $\psi_1^{-1}(V_G /\aug{G}^2)=H_t^{\mathfrak d}(G)\Phi(G)/\Phi(G)$, and hence \eqref{eq:directproduct} follows from \Cref{lemma:complement}.
  Conversely, assume that \eqref{eq:directproduct} holds. Let  $\{x_1,\dots, x_r\}$ be a (Burnside) basis of $ \psi_1^{-1}(V_G/\aug {G}^2) $. Take elements $y_i\in \Omega_t(\ZZ(G))$ such that  $y_i\Phi(G)=x_i$ for $1\leq i\leq r$, and set $T=\GEN{y_1,\dots, y_r} $. Then $d(T)=r=d(T\Phi(G)/\Phi(G))$ and $\psi_1^{-1}(V_G /\aug{G}^2)=T\Phi(G)/\Phi(G)$. Therefore \Cref{lemma:complement} yields that there is a  decomposition $\mathfrak d$ of $G$ such that 
  $T=H_t^{\mathfrak d}(G)$, so that
  $ V_G^t=\augNor{T}{G}+\aug{G}^2=\augNor{H_t^{\mathfrak d}(G)}{G}+\aug{G}^2 $ 
  as desired.

\end{proof}

\begin{lemma}\label{lemma:pretheorem}
	Let $\phi:kG\to kH$ be an isomorphism preserving the augmentation, and $\mathfrak d$ be a  homocyclic decomposition of $G$. Then there exists a  homocyclic decomposition $\bar {\mathfrak d}$ of $H$ such that if $H=S\oplus H_t^{\bar{\mathfrak d}}(H)$, then  $kH=\phi(\augNor{H_t^{\mathfrak d}(G)}{G})\oplus kS$.
\end{lemma}
\begin{proof}  Set $V_G = \augNor{H_t^{\mathfrak d}(G)}{G}+\aug{G}^2$, so that by \Cref{subspaceV} we have that 
	$$\frac{V_G }{\aug{G}^2}\oplus \ker\Lambda_G^{t-1}=\frac{\augNor{\Omega_t(\ZZ(G))G'}{G}+\aug{G}^2}{\aug{G}^2}.  $$
It follows	from  \eqref{Salim1}  and \eqref{eq:example2}   that $\phi$ induces isomorphisms $\tilde \phi$ and $\hat\phi$ such that the following diagram commutes:
	\begin{equation}\label{diagram3}\begin{array}{c}
		\xymatrix{
			\frac{\augNor{\Omega_t(\ZZ(G))G'}{G}+\aug{G}^2}{\aug{G}^2} \ar[d]_-{\tilde \phi} \ar[r]^-{\Lambda_G^{t-1}} & \frac{\aug{G}^{p^{t-1}}}{\aug{G}^{p^{t-1}+1}+\aug{\mho_t(G)G'}{G} } \ar[d]^-{\hat\phi} \\
			\frac{\augNor{\Omega_t(\ZZ(H))H'}{H}+\aug{H}^2}{\aug{H}^2} \ar[r]_-{\Lambda_H^{t-1}} & \frac{\aug{H}^{p^{t-1}}}{\aug{H}^{p^{t-1}+1}+\aug{\mho_t(H)H'}{H} }
		} 
\end{array}
	\end{equation}
	In particular this shows that $\tilde \phi(\ker \Lambda_G^{t-1})=\ker \Lambda_H^{t-1}$, and therefore that
	$$\frac{\phi(V_G)}{\aug{H}^2}\oplus \ker \Lambda_H^{t-1}=	\frac{\augNor{\Omega_t(\ZZ(H))H'}{H}+\aug{H}^2}{\aug{H}^2}.$$   Applying once more \Cref{subspaceV}, we derive the existence of a homocyclic  decomposition $\bar {\mathfrak d}$ of $H$ verifying
	$$\phi(V_G)=\augNor{H_t^{\bar {\mathfrak d}}(H)}{H}+\aug{H}^2 . $$  
	Now $\phi(\augNor{H_t^{\mathfrak d}(G)}{G})$ is an ideal of $kH$ verifying the following conditions:
	\begin{enumerate}
		\item $\text{codim}_{kH}(\phi(\augNor{H_t^{\mathfrak d}(G)}{G}))=\text{codim}_{kG}(\augNor{H_t^{\mathfrak d}(G)}{G})=|G|/|H_t^{\mathfrak d} (G)|=|H|/|H_t^{\bar{\mathfrak d}}(H)| $, where the last equality is due to   \Cref{lemma:abelianfactor};
		\item $\phi(\augNor{H_t^{\mathfrak d}(G)}{G})+\aug{H}^2=\phi\left(\augNor{H_t^{\mathfrak d}(G)}{G}+\aug{G}^2 \right)=\phi(V_G)=\augNor{H_t^{\bar{\mathfrak d}}(H)}{H}+ \aug{H}^2.$
	\end{enumerate}
Thus   \Cref{lemma:semidirect2} yields that for any subgroup $S$ of $H$ such that $H=S\oplus H_t^{\bar{\mathfrak d}}(H)$ one has that $kH=\phi(\augNor{H_t^{\mathfrak d}(G)}{G}) \oplus kS$, as desired.
\end{proof}

Now \Cref{maintheorem} follows easily: 
\begin{theorem}\label{theorem:reduction}
	Let $G$ and $H$ be finite $p$-groups  and $k=\F_p$. Then   \begin{equation*}kG\cong kH \qquad \text{\rm if and only if}\qquad 
		k\left(\textrm{\rm NAb}(G)\right)\cong k\left( \textrm{\rm NAb}(H)\right) \qand  \textrm{\rm Ab}(G)\cong\textrm{\rm Ab}(H).
	\end{equation*}
\end{theorem}
\begin{proof}
	If $	k\left(\textrm{\rm NAb}(G)\right)\cong k\left( \textrm{\rm NAb}(H)\right)$ and $ \textrm{\rm Ab}(G)\cong\textrm{\rm Ab}(H)$, then also $k\left(\text{\rm Ab}(G)\right)\cong k\left(\text{\rm Ab}(H)\right)$, so that 
	$$kG\cong  k\left(\textrm{\rm NAb}(G)\right)\otimes_k k\left(\text{\rm Ab}(G)\right)\cong k\left( \textrm{\rm NAb}(H)\right)\otimes_k k\left(\text{\rm Ab}(H)\right)\cong kH. $$
	Conversely, assume that $\phi:kG\to kH$ is an isomorphism preserving the augmentation. Then by \Cref{lemma:abelianfactor} we have that $\text{\rm Ab}(G)\cong\text{\rm Ab}(H)$. Consider a sequence $(\mathcal G_i, \mathcal H_i,\phi_i)_{i}$  defined recursively as follows. Set $ \mathcal G_0= G$, $ \mathcal H_0= H$ and $\phi_0=\phi$.
	Assume that for $i\geq 0$ the tuple $(\mathcal G_i,\mathcal H_i,\phi_i)$ is defined such that $\phi_i:k\mathcal G_i\to k\mathcal H_i$ is an isomorphism preserving the augmentation. Choose a  homocyclic decomposition $\mathfrak d_i $ of $\mathcal G_i$ and take a subgroup $\mathcal G_{i+1}$ of $\mathcal G_i$ such that $\mathcal G_i=\mathcal G_{i+1}\oplus H_i^{\mathfrak d_i}(\mathcal G_i)$; then by \Cref{lemma:pretheorem} there is a  homocyclic decomposition $\bar{\mathfrak d}_i$ of $\mathcal H_i$ such that, given a subgroup $\mathcal H_{i+1}$ of $\mathcal H_i$ satisfying $\mathcal H_i=\mathcal H_{i+1}\oplus H_i^{\bar {\mathfrak d}_i}(\mathcal H_i)$, one has that
	$$k\mathcal H_i=\phi(\augNor{ H_i^{\mathfrak d_i}(\mathcal G_i)}{\mathcal G_i}) \oplus  k\mathcal H_{i+1}.$$
	Since $k\mathcal G_i=\augNor{H_i^{\mathfrak d_i}(\mathcal G_i)}{\mathcal G_i}\oplus k\mathcal G_{i+1}$ by \Cref{lemma:semidirect1} for $n=1$, it follows that $\phi_i$ induces an isomorphism
	$$\phi_{i+1}:k\mathcal G_{i+1}\to k\mathcal H_{i+1},$$ and we can advance to the next step of the sequence. If $i$ is large enough (for example if $p^i$ exceeds the exponent of $G$), by the Krull-Remak-Schmidt theorem    $\mathcal G_i\cong \text{\rm NAb}(G)$ and $\mathcal H_i\cong \text{\rm NAb}(H)$. Thus we have an isomorphism $\phi_i:k\left(\text{\rm NAb}(G)\right) \to k\left(\text{\rm NAb}(H)\right)$ as desired.
\end{proof}

\begin{remarks}\rm 
	Observe that both \Cref{mainprop} and \Cref{maintheorem} rely on the results in \Cref{section:transfer}: indeed, the former depends on the fact that the subgroup assignations $G\mapsto \Omega_t(\ZZ(G))\mho_t(G)G'$ and $G\mapsto \mho_t(G)G'$ are $k$-canonical to show that the isomorphism type of the quotient $\Omega_t(\ZZ(G))\mho_t(G)G'/ \mho_t(G)G'$ is determined by $kG$ in \Cref{lemma:abelianfactor}, while the latter, in addition, uses that $G\mapsto \Omega_t(\ZZ(G))G'$ is $k$-canonical to guarantee that the arguments involving the commutative diagram \eqref{diagram3} make sense. 
\end{remarks}

In the light of \Cref{maintheorem},  \Cref{maincorollary} follows easily. Indeed, it suffices to observe that if a non-abelian finite $p$-group $G$ satisfies one of the properties \eqref{maincorollary1}-\eqref{maincorollary8} in \Cref{maincorollary}, so does $\text{\rm NAb}(G)$, and then apply \Cref{prop:knowncases}. 

\vspace*{0.2cm}

\noindent\textbf{Acknowledgements}. The author  wish to express his  gratitude to Ángel del Río and Mima Stanojkovski for fruitful conversations on the ideas in this paper  and   for their generosity. To the former and to Leo Margolis also for their very useful comments   and careful reading of earlier versions of this paper.

 \bibliographystyle{amsalpha}
 \bibliography{MIP}

\end{document}